\documentclass{scrartcl}
\usepackage{geometry}
\usepackage{amsmath,amsthm,amsfonts,amssymb,amscd}
\usepackage{mathtools}
\usepackage{mathdots}

\usepackage{booktabs}
\usepackage[capbesideposition=inside,facing=yes,capbesidesep=quad]{floatrow}

\usepackage{tikz}

\usepackage[T2A]{fontenc}
\usepackage[utf8]{inputenc}
\usepackage[russian,english]{babel}
\addtokomafont{disposition}{\rmfamily}

\usepackage[capitalise,noabbrev]{cleveref}

\newtheorem*{theorem*}{Theorem}
\newtheorem{lemma}{Lemma}

\newtheorem*{corollary*}{Corollary}

\newtheorem*{proposition*}{Proposition}

\newtheorem*{remark*}{Remark}

\newtheorem*{definition*}{Definition}

\newcommand{\rB}{\mathsf{B}}
\newcommand{\rC}{\mathsf{C}}

\newcommand{\rF}{\mathsf{F}}
\newcommand{\rG}{\mathsf{G}}

\DeclareMathOperator{\Sp}{Sp}
\DeclareMathOperator{\Spin}{Spin}
\DeclareMathOperator{\SCliff}{SCliff}
\DeclareMathOperator{\End}{End}

\tikzset{dynkin-node/.style={draw,circle,inner sep=2}}

\title{Suzuki---Ree groups and\\Tits mixed groups over rings}
\author{Andrei Smolensky\thanks{\,\ email: \texttt{andrei.smolensky@gmail.com}\newline Department of Mathematics and Mechanics, Saint Petersburg State University\newline The research was supported by RSF (project No. 17-11-01261)}}

\begin{document}
\maketitle
\begin{abstract}
It is shown that Suzuki---Ree groups can be easily defined by means of comparing two fundamental representations of the ambient Chevalley group in characteristic $2$ or $3$. This eliminates the distinction between the Suzuki---Ree groups over perfect and imperfect fields and gives a natural definition for the analogues of such groups over commutative rings. As an application of the same idea, we explicitly construct a pair of polynomial maps between the groups of types $\rB_n$ and $\rC_n$ in characteristic $2$ that compose to the Frobenius endomorphism. This, in turn, provides a simple definition for the Tits mixed groups over rings.\end{abstract}

The classical definition of Suzuki---Ree groups over a perfect field \cite{TitsSuzukiRee} uses the existence of the exceptional root length changing automorphism of the Chevalley group of type $\rC_2$, $\rG_2$ or $\rF_4$ and the endomorphism of the base field that squares to Frobenius. Over a ring, the former is only defined for the elementary subgroup, while the latter need not be invertible. Below we give a new definition for Suzuki---Ree groups that incorporates the case of imperfect fields \cite[II.10]{TitsOvoides} and works for commutative rings.

Let $\Phi$ be a root system of type $\rC_2$, $\rG_2$ or $\rF_4$, and let $\sigma$ be a mapping induced by the symmetry of its Dynkin diagram. Consider the fundamental weights $\lambda$ and $\mu$ corresponding to the terminal nodes of the Dynkin diagram, the latter are permuted by $\sigma$. Assume that $\lambda$ is the highest weights of the minimal representation of $G(\Phi)$, then $\mu$ is the highest weight of the adjoint representation in case $\Phi=\rG_2, \rF_4$ and of the unique short-roots representation in case $\Phi=\rC_2$ (\cref{table:representations}). It so happens that in characteristic $2$ (for $\rC_2$ and $\rF_4$) or $3$ (for $\rG_2$) the dimension of the irreducible highest weight representation $V(\mu)$ drops and becomes equal $\dim V(\lambda)$.
\begin{table}[h]
\centering
\begin{tabular}{cccllll}
\toprule
$\Phi$ & $p$ & roots numbering & $\lambda$ & $\mu$ & $V(\lambda)_\mathbb{Z}$ & $V(\mu)_\mathbb{Z}$ \\
\midrule
$\rC_2$ & $2$ &
\includegraphics{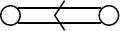}
& $\varpi_1$ & $\varpi_2$ & minimal, $4$-dim & short-roots, $5$-dim \\
$\rG_2$ & $3$ &
\includegraphics{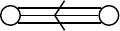}
& $\varpi_1$ & $\varpi_2$ & minimal, $7$-dim & adjoint, $14$-dim \\
$\rF_4$ & $2$ &
\includegraphics{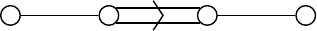}
& $\varpi_4$ & $\varpi_1$ & minimal, $26$-dim & adjoint, $52$-dim \\
\bottomrule
\end{tabular}
\caption{Weights and representations}
\label{table:representations}
\end{table}

In what follows a subscript $p$ will denote the restriction to characteristic $p$. Thus $G_p$ means the affine group scheme over $\mathbb{F}_p$ obtained in a natural way from a $\mathbb{Z}$-scheme $G$, and $V(\lambda)_p$ denotes the $G_p$-module obtained from $G$-module $V(\lambda)_\mathbb{Z}$. The module $V(\lambda)_p$ might not be irreducible, so we will denote by $V(\lambda)$ the irreducible $G_p$-module with the highest weight $\lambda$.

To see that $\dim V(\lambda) = \dim V(\mu) < \dim V(\mu)_p$, let us construct $V(\mu)$ explicitly as a submodule or a quotient of $V(\mu)_p$. This amounts to finding the unique maximal proper $G_p$-invariant submodule of $V(\mu)_p$.

If $\Phi=\rC_2$, the required $G(\rC_2)_2$-invariant submodule of $V(\varpi_2)_2=\langle e_1, e_2, e_0, e_{-2}, e_{-1}\rangle$ is the symplectic subspace $\langle e_1,e_2,e_{-2}, e_{-1}\rangle$ of codimension $1$.

If $\Phi=\rG_2, \rF_4$, the short-roots Chevalley generators $\{e_\alpha \mid \alpha\in\Phi^{<}\}$ together with $\{h_i\mid \alpha_i\in\Pi^{<}\}$ generate an ideal $\mathfrak{s}$ of $\mathfrak{g}(\Phi)_p$, which is the subrepresentation of $V(\mu)_p$. Note that $\dim\mathfrak{s}=\dim V(\mu)_\mathbb{Z}-\dim V(\lambda)_\mathbb{Z}$.

Let now $R$ be a commutative ring of characteristic $p$ and let $\tau$ be a Tits endomorphism on $R$, that is, $\tau^2=\varphi$, the Frobenius endomorphism. $\tau$ induces a mapping on every free $R$-module, in particular, on $V(\lambda)_R$ and $V(\mu)_R$. We identify these two $R$-modules by means of choice of basises (note that the weights of both representations are isomorphic as posets, see~\cref{fig:wd-adjoint}). For an element $g\in G(\Phi, R)$ we will denote by $g_\nu$ the corresponding operator on $V(\nu)_R$.

\begin{figure}[h]
\centering
\includegraphics{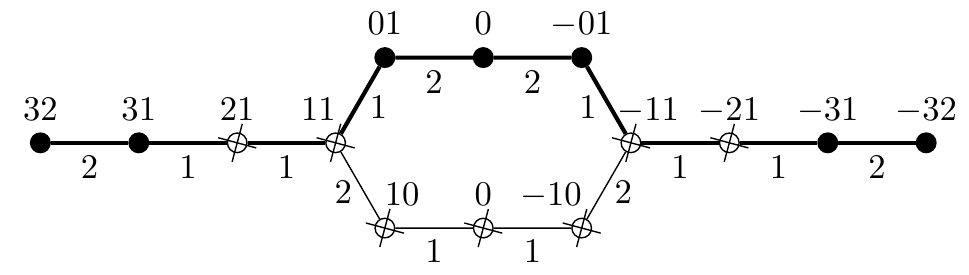}

\vspace{1em}
\includegraphics{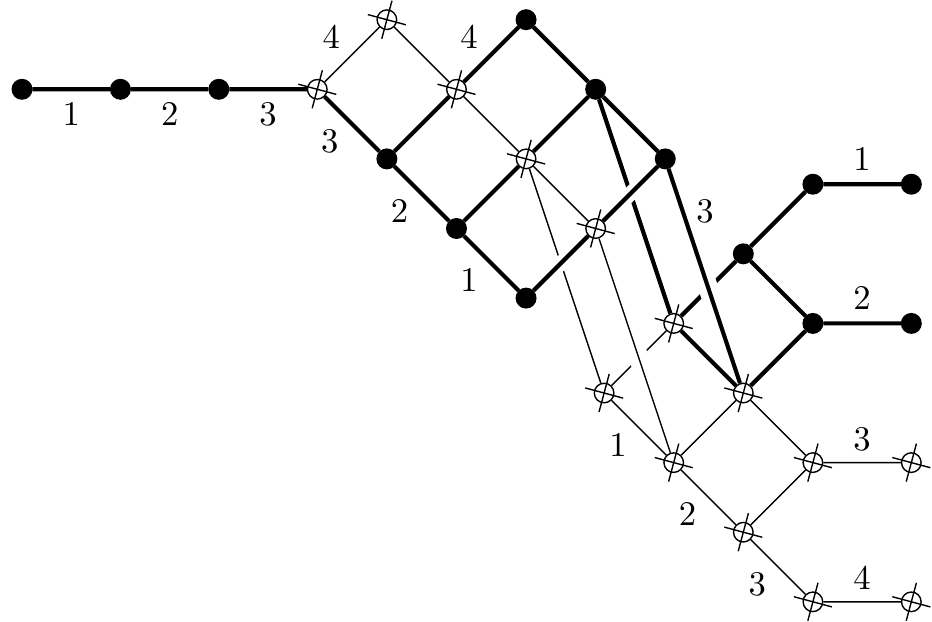}
\caption{Weight diagrams of the adjoint representations of $\rG_2$ and $\rF_4$ (in the latter only the positive part is shown), see~\cite{PloSemVavVBRA} for details. The crossed out vertices correspond to the basis of $\mathfrak{s}$ and are quotiented out. The vertices shown in black and joined by the bold edges survive and form the weight diagram of the minimal representation.}
\label{fig:wd-adjoint}
\end{figure}

We define Suzuki---Ree groups as
\[ {}^\sigma G(\Phi, R, \tau) = \{ g\in G(\Phi, R) \mid \tau g_\lambda = g_\mu \tau \}. \]

This set is obviously closed under multiplication. To show that it is closed under taking inverses consider the equality $\tau=g_\mu^{-1}g_\mu\tau$, rewrite the right-hand side as $g_\mu^{-1}\tau g_\lambda$ and multiply both sides by $g_\lambda^{-1}$ from the right. Define
\begin{align*}
U & = {}^\sigma U(\Phi, R, \tau) = U(\Phi, R)\cap{}^\sigma G(\Phi, R, \tau), \\
H & = {}^\sigma H(\Phi, R, \tau) = H(\Phi, R)\cap{}^\sigma G(\Phi, R, \tau), \\
B & = {}^\sigma B(\Phi, R, \tau) = B(\Phi, R)\cap{}^\sigma G(\Phi, R, \tau) = HU, \\
N & = {}^\sigma N(\Phi, R, \tau) = N(\Phi, R)\cap{}^\sigma G(\Phi, R, \tau).
\end{align*} 

\begin{lemma}\label{lemma:Bruhat}
If $F$ is a field, Suzuki---Ree group admits Bruhat decomposition
\[ {}^\sigma G(\Phi, F, \tau) = UNU. \]
\end{lemma}
\begin{proof}
For an element $g\in {}^\sigma G(\Phi, F, \tau)$ let $g=uhwv$ be its restricted Bruhat decomposition in $G(\Phi, F)$, that is, $u\in U(\Phi, F)$, $h\in H(\Phi, F)$, $w\in N(\Phi,\mathbb{F}_2)$, $v\in U(\Phi, F)\cap {}^w U^-(\Phi, F)$. Then $\tau u_\lambda h_\lambda w_\lambda v_\lambda = u_\mu h_\mu w_\mu v_\mu \tau$. Rewrite this as
\[ \tau u_\lambda h_\lambda w_\lambda = u_\mu h_\mu w_\mu v_\mu \tau v_\lambda^{-1}\]
and note that for any (uni-)triangular matrix $u$ there exists a (uni-)triangular matrix $u'$ such that $\tau u=u'\tau$ (namely, $u'$ is obtained from $u$ by the element-wise application of $\tau$). %Thus
\[ u'h'\tau w_\lambda = u_\mu h_\mu w_\mu v_\mu v' \tau,\quad\text{and hence}\quad \tau w_\lambda = h''u''u_\mu h_\mu w_\mu v_\mu v' \tau. \]
Everything is $\mathbb{F}_p$-linear, the operator in the left-hand side acts naturally on $V(\lambda, \mathbb{F}_p)$, an $\mathbb{F}_p$-submodule of $V(\lambda)$, and so the same applies to the one in the right-hand side. Moreover, in this action $\tau$ can be discarded, for $\tau|_{\mathbb{F}_p}=\operatorname{id}$. Thus both operators $w_\lambda = h''u''u_\mu h_\mu w_\mu v_\mu v'$ are written in terms of their restricted Bruhat decompositions in $G(\Phi, F)$, which is unique, and thus $w_\lambda=w_\mu$, $u''u_\mu=v_\mu v'=h''h_\mu=1$. Restoring $u_\lambda$, $h_\lambda$ and $v_\lambda$ from $u''$, $h''$ and $v'$, one sees that $g=uhwv$ is the Bruhat decomposition in ${}^\sigma G(\Phi, F, \tau)$.
\end{proof}

\section{Suzuki group, the case $\Phi=\rC_2$.}
It is easy to check that
\[ x_\gamma(\xi)_\mu =
\begin{cases}
x_{\sigma(\gamma)}(\xi), & \text{if $\gamma$ is long,} \\
x_{\sigma(\gamma)}(\xi^2), & \text{if $\gamma$ is short.}
\end{cases} \]
One can now describe the structure of $U$ by comparing for $u=\prod_{\gamma\in\rC_2^+}x_\gamma(\xi_\gamma)\in U$ the action of $\tau u_\lambda$ and $u_\mu\tau$ on $e_\nu$, where $\nu$ is the lowest weight, and solving for $\xi_\gamma$.  Namely,
\begin{gather*}
{}^\sigma U(\rC_2, R, \tau) \ni u = x_+(a, b) = x_{\alpha}(a) x_{\beta}(a^\tau) x_{\alpha+\beta}(b) x_{2\alpha+\beta}(a^{\tau+2}+b^\tau), \\
x_+(a, b)_\lambda =
\begin{pmatrix}
1 & a & b+a^{\tau+1} & ab+b^\tau+a^{\tau+2} \\
& 1 & a^\tau & b \\
&& 1 & a \\
&&& 1
\end{pmatrix}.
\end{gather*}
Recall that an element of the torus $H(\rC_2, R)$ is of the form $h=h_\alpha(\varepsilon_1)h_\beta(\varepsilon_2)$, where
\[ h_\gamma(\varepsilon) = w_\gamma(\varepsilon)w_{\gamma}(-1),\quad w_\gamma(\varepsilon)=x_\gamma(\varepsilon)x_{-\gamma}(-\varepsilon^{-1})x_\gamma(\varepsilon). \]
Thus $h_\lambda=\operatorname{diag}(\varepsilon_1,\varepsilon_2/\varepsilon_1,\varepsilon_1/\varepsilon_2,1/\varepsilon_1)$ and $h_\mu=\operatorname{diag}(\varepsilon_2,\varepsilon_1^2/\varepsilon_2,\varepsilon_2/\varepsilon_1^2,1/\varepsilon_2)$, and so $h\in H$ if $\varepsilon_2=\varepsilon_1^\tau$ and $h$ is of the form $h(\varepsilon)_\lambda=\operatorname{diag}(\varepsilon,\varepsilon^{\tau-1},\varepsilon^{1-\tau},\varepsilon^{-1})$.

Note the following relations:
\begin{align*}
& x_+(a,b)x_+(c,d) = x_+(a+c,b+d+a^\tau c), \\
& {}^{h(\varepsilon)}x_+(a,b) = x_+(\varepsilon^{2-\tau}a, \varepsilon^\tau b), \\
& [h(\varepsilon), x_+(0,b)] = x_+(0,b+\varepsilon^\tau b), \\
& [h(\varepsilon), x_+(a,0)] = x_+(a+\varepsilon^{2-\tau}, a^{\tau+1}(\varepsilon^{2\tau-2}+1)).
\end{align*}
The torus normalizer $N$ equals $H\cup\widehat{w_0}H$, where $\widehat{w_0}=\operatorname{antidiag}(1,1,1,1)$. Note that
\[ \widehat{w_0} h(\varepsilon) = x_-(0,\varepsilon)^{x_+(\varepsilon^{1-\tau}, 0)}, \]
and so over a field $F$ it follows from Bruhat decomposition that ${}^\sigma G(\rC_2, F, \tau) = \langle U, U^- \rangle$. Here 
\[ U^-={}^{\widehat{w_0}}U=\{x_-(a,b)\mid a,b\in R \}, \quad x_-(a,b)={}^{\widehat{w_0}}x_+(a,b). \]
If $F\neq\mathbb{F}_2$, the Suzuki group ${}^\sigma G(\rC_2, F, \tau)$ is perfect. Indeed, for $\varepsilon\in F\setminus\{0, 1\}$ one has
\begin{align*}
& [h(\varepsilon), x_+(0, b/(\varepsilon^\tau+1))] = x_+(0, b), \\
& [h(\varepsilon), x_+(a/(\varepsilon^{2-\tau}+1), 0)] = x_+(a, a^{1+\tau}/(\varepsilon^{2-\tau}+1)).
\end{align*}

\section{Small Ree group, the case $\Phi=\rG_2$.}
In this case one has to fix the signs of the structure constants appropriately in order to obtain a nice description for ${}^\sigma U(\rG_2, R, \tau)$. We construct the Lie algebra and Chevalley group of type $\rG_2$ in their $7$-dimensional representation as simultaneously preserving certain symmetric bilinear and alternating trilinear forms, defined as follows.

Fix a basis $e_i$, $i=1,2,3,0,-3,-2,-1$ of $V$, then let $B$ be a bilinear form with Gram matrix $B_e=\operatorname{antidiag}(1,1,1,2,1,1,1)$. Define $T$ to be the unique alternating trilinear form such that $T(e_i, e_j, e_k) = 1$ if $(i,j,k)$ is one of $(0,1,-1)$, $(0,-2,2)$, $(0,-3,3)$, $(1,-2,-3)$ or $(-1,3,2)$ and $0$ if $(i,j,k)$ is not the permutation of one of the triples above. Note that this description differs from the classical and more symmetric formula for the Dickson form \cite{AschbG2Trilinear}, \cite[section 5]{AschbMultilinear}, but suits better for our purposes.

Now we define the $\rG_2$ Lie algebra and Chevalley group as
\begin{align*}
& \mathfrak{g}(\rG_2) = \left\{ g\in\mathfrak{gl}_7\ \middle|\ \forall u,v,w\in V \quad \begin{aligned}
& B(gu,v)+B(u,gv)=0,\\
& T(gu,v,w)+T(u,gv,w)+T(u,v,gw)=0
\end{aligned} \right\}, \\
& G(\rG_2) = \left\{ g\in GL_7\ \middle|\ \forall u,v,w\in V \quad
\begin{aligned}
& B(gu,gv) = B(u,v), \\
& T(gu,gv,gw) = T(u,v,w)
\end{aligned} \right\}.
\end{align*}

Let $\{ \alpha, \beta, \alpha+\beta, 2\alpha+\beta, 3\alpha+\beta, 3\alpha+2\beta \}$ be the set of positive $\rG_2$ roots. We fix a Chevalley basis in $\mathfrak{g}(\rG_2)$ as follows:
\begin{align*}
&\mathsf{e}_\alpha = e_{12}-2e_{3,0}+e_{0,-3}-e_{-2,-1}, \\
&\mathsf{e}_\beta = e_{23}-e_{-3,-2}, \\
&\mathsf{e}_{\alpha+\beta} = e_{13}+2e_{2,0}-e_{0,-2}-e_{-3,-1}, \\
&\mathsf{e}_{2\alpha+\beta} = 2e_{1,0}-e_{2,-3}+e_{3,-2}-e_{0,-1}, \\
&\mathsf{e}_{3\alpha+\beta} = -e_{1,-3}+e_{3,-1}, \\
&\mathsf{e}_{3\alpha+2\beta} = -e_{1,-2}+e_{2,-1}, \\
&\mathsf{e}_{\gamma} = -P\mathsf{e}_{-\gamma}P, \quad \gamma\in\Phi^-, \\
&\mathsf{h}_\alpha = [\mathsf{e}_\alpha,\mathsf{e}_{-\alpha}] = \operatorname{diag}(1,-1,2,0,-2,1,-1), \\
&\mathsf{h}_\beta = [\mathsf{e}_\beta,\mathsf{e}_{-\beta}] = \operatorname{diag}(0,1,-1,0,1,-1,0).
\end{align*}
Here $P=\operatorname{antidiag}(1,\ldots,1)$ is the peridentity matrix. Note that the choice of signs is not compatible with \cite[Lemma~12.4.2]{CarterLie}. We define the elementary generators as
\[ x_\gamma(\xi) = \exp(\xi\mathsf{e}_\gamma),\quad \gamma\in\Phi. \]
A straightforward computation shows that
\[ x_\gamma(\xi)_\mu = x_{\sigma(\gamma)}\left((-1)^{1+\operatorname{ht}(\gamma)}\cdot\xi^{e(\gamma)}\right)_{\!\lambda}, \quad \text{where} \quad e(\gamma) =
\begin{cases}
1, & \text{if $\gamma$ is long,} \\
3, & \text{if $\gamma$ is short.}
\end{cases} \]

An element of $U(\rG_2, R)$ lies in ${}^2 G(\rG_2, R, \tau)$ if it is of the form
\begin{align*}
x_+(a,b,c) & =x_1(a)x_2(b)x_3(c),\quad \text{where}\ a,b,c\in R,\ \text{and} \\
x_1(a) & = x_\alpha(a)x_\beta(a^\tau)x_{\alpha+\beta}(-a^{\tau+1})x_{2\alpha+\beta}(a^{\tau+2}), \\
x_2(b) & = x_{\alpha+\beta}(b)x_{3\alpha+\beta}(-b^\tau), \\
x_3(c) & = x_{2\alpha+\beta}(c)x_{3\alpha+2\beta}(c^\tau).
\end{align*}
The elements of the torus in $G(\rG_2)$ are of the form $h=h_\alpha(\varepsilon_1)h_\beta(\varepsilon_2)$, where
\begin{align*}
& h_\alpha(\varepsilon)_\lambda = \operatorname{diag}(\varepsilon, 1/\varepsilon, \varepsilon^2, 1, 1/\varepsilon^2, \varepsilon, 1/\varepsilon), &&
h_\beta(\varepsilon)_\lambda = \operatorname{diag}(1, \varepsilon, 1/\varepsilon, 1, \varepsilon, 1/\varepsilon, 1), \\
& h_\alpha(\varepsilon)_\mu = \operatorname{diag}(1, \varepsilon^3, 1/\varepsilon^3, 1, \varepsilon^3, 1/\varepsilon^3, 1), &&
h_\beta(\varepsilon)_\mu = \operatorname{diag}(\varepsilon, 1/\varepsilon, \varepsilon^2, 1, 1/\varepsilon^2, \varepsilon, 1/\varepsilon).
\end{align*}
From this one easily deduces that $h\in H$ if $\varepsilon_2=\varepsilon_1^\tau$ and $h$ is of the form
\[ h(\varepsilon)_\lambda=\operatorname{diag}(\varepsilon, \varepsilon^{\tau-1}, \varepsilon^{2-\tau}, 1, \varepsilon^{\tau-2}, \varepsilon^{1-\tau}, \varepsilon^{-1}). \]
The relations between these elements are
\begin{align*}
& \begin{multlined}
x_+(a_1,b_1,c_1)x_+(a_2,b_2,c_2) = \\ \qquad = x_+(a_1+a_2, b_1+b_2+a_1a_2^\tau, c_1+c_2+b_1a_2+a_1a_2^{\tau+1}-a_1^2a_2^\tau),
\end{multlined} \\
& x_+(a,b,c)^{-1} = x_+(-a,-b+a^{\tau+1},-c+ab+a^{\tau+2}), \\
& {}^{h(\varepsilon)} x_+(a,b,c) = x_+(\varepsilon^{2-\tau}a, \varepsilon^{\tau-1}b, \varepsilon c), \\
& [h(\varepsilon), x_+(0,0,c)] = x_+(0,0,c(\varepsilon-1)), \\
& [h(\varepsilon), x_+(0,b,0)] = x_+(0,b(\varepsilon^{\tau-1}-1),0), \\
& [h(\varepsilon), x_+(a,0,0)] = x_+(a(\varepsilon^{2-\tau}-1), a^{\tau+1}(1-\varepsilon^{2-\tau}), a^{\tau+2}(\varepsilon^{2-\tau}-1)^2).
\end{align*}
The torus normalizer $N$ equals $H\cup\widehat{w_0}H$, where $\widehat{w_0}=\operatorname{antidiag}(-1,\ldots,-1)$. We also define
\[ U^-={}^{\widehat{w_0}}U=\{x_-(a,b,c)\mid a,b,c\in R \}, \quad x_-(a,b,c)={}^{\widehat{w_0}}x_+(a,b,c). \]

It is not known whether ${}^\sigma G(\rG_2, F, \tau)$ is perfect when $F$ is an infinite field, but it is always quasi-perfect, with the sole exception $F=\mathbb{F}_3$. Denote by $E$ the elementary subgroup of $G={}^\sigma G(\rG_2, F, \tau)$, that is, $E=\langle U,U^- \rangle$.
\begin{lemma}
If $F\neq\mathbb{F}_3$, then $[G,G]=E$ and is perfect.
\end{lemma}
\begin{proof}
Since $G=\langle U, N \rangle$, one has $[G,G]= \langle [f,g] \mid f,g\in U\cup N \rangle^G$. Using the relations listed above, it is easy to show that $E\leqslant[G,G]$. Namely,
\begin{align*}
& x_+(0,0,c) = [h(-1), x_+(0,0,c)], \\
& x_+(0,b,0) = [h(\varepsilon), x_+(0,b/(\varepsilon^{\tau-1}-1),0)] \quad \text{for } \varepsilon\notin\mathbb{F}_3, \\
& x_+(a,0,0) = [h(-1), x_+(a,0,0)]\cdot x_+(0,*,*).
\end{align*}
On the other hand, all the commutators $[f,g]$ with $f,g\in U\cup N$ lie in $E$. The inclusion $[N,U]\leqslant E$ follows from the expression for ${}^{h(\varepsilon)}x_+(a,b,c)$ and the definition of $U^-$. It only remains to consider commutators of the form $[N,N]$. But $H$ is abelian, while
\[ [h(\varepsilon),\widehat{w_0}h(\eta)] = h(\varepsilon^2) \quad \text{and} \quad [\widehat{w_0}h(\varepsilon),\widehat{w_0}h(\eta)]=h(\eta^2/\varepsilon^2), \]
so $[N, N] = \{ h(\varepsilon^2) \mid \varepsilon\in R^* \} = H^{(2)}$. The coset $\widehat{w_0}H^{(2)}$ is contained in $E$ because
\[ x_+(0,1/\eta,0)x_-(0,\eta,0)x_+(0,1/\eta,0) = \widehat{w_0} h(\eta^{\tau+1}), \]
and substituting $\eta=\varepsilon^{\tau-1}$ gives $\eta^{\tau+1}=\varepsilon^2$. Now $H^{(2)}=\left( \widehat{w_0}H^{(2)}\right)^2$.

It follows from Bruhat decomposition and the inclusion ${}^NU\leqslant E$ that $E$ is a normal subgroup of $G$. Since it contains the normal generators of $[G,G]$, one concludes that $E\geqslant [G,G]$.

To show that $E$ is perfect, note that in the commutator expressions for $x_+(a,b,c)$ all factors can be taken from $E$. Indeed,
\[ x_+(-\eta^{-1}, \eta^{-1-\tau}, \eta^{-2-\tau}) x_-(\eta,0,0) x_+(\eta^{-1},-\eta^{-1-\tau},\eta^{-2-\tau}) = \widehat{w_0} h(-\eta^{4+2\tau}) \]
and substituting $\eta=\varepsilon^{2-\tau}$ gives $-\eta^{4+2\tau}=-\varepsilon^2$, in particular, $h(-1)\in E$.
\end{proof}
\begin{remark*}
If the field $F$ is finite, $G=E$.
\end{remark*}
\begin{proof}
If $F$ is finite and $\operatorname{char}(F)=3$, then $F^*=\{ \pm\varepsilon^2 \mid \varepsilon\in F^* \}$, so $H\leqslant E$.
%In a finite field of characteristic $3$ every element is either a square or a minus square, so by the fornula above $H\leqslant E$.
\end{proof}

\section{Simplicity}
\begin{lemma}
If $F\neq\mathbb{F}_p$ is a field, then the commutator subgroup $G={}^\sigma\!G(\Phi, F, \tau)'$ is simple.
\end{lemma}
\begin{proof}
Since $\widetilde{B}=(H\cap G)U$ and $N\cap G$ form a $BN$-pair for $G$, we will prove the simplicity by appealing to a general result of Tits (see~\cite[Theorem~11.1.1]{CarterLie}). Namely, it suffices to show that $\widetilde{B}$ is solvable and core-free, the ambient group $G$ is perfect and that the simple reflections cannot be divided into two commuting subsets. The solvability of $\widetilde{B}$ follows from the solvability of the group of upper triangular matrices, the perfectness of $G$ has been proved above, and the Weyl group property is easy to check. It remains to show that the intersection of all conjugates of $\widetilde{B}$ is trivial. Note that $\widehat{w_0} B\widehat{w_0}^{-1}\cap B = H$. Now consider ${}^{v}B$ for some $v\in U^-$, $v\neq1$. Its elements are of the form $vh(\varepsilon)uv^{-1}$ for some $t\in F^*$ and $u\in U$. If this element lies in $H$, then $vh(\varepsilon)u=h(\eta)v=v'h(\eta)$ for some $\eta\in F^*$ and $v'\in U^-$. By the uniqueness of the Bruhat decomposition $u=1$ and $\varepsilon=\eta$. But $v\neq v'$ unless $\varepsilon=1$, and so $B\cap{}^{\widehat{w_0}}B\cap{}^v B=1$.
\end{proof}

\section{Polynomial mappings between $\rB_n$ and $\rC_n$}
As an application of the same idea, we explicitly construct a pair of polynomial maps between the groups of type $\rC_n$ and $\rB_n$ in characteristic $2$ that compose to the Frobenius endomorphism. The existence of such maps allowed Nuzhin and Stepanov \cite{NuzSteBnCnSubring} to carry the result of Bak and Stepanov about the subring subgroups of symplectic groups \cite{BakSteSymplecticSubring} to the subgroups of the Chevalley group $G(\rB_n,A)$ that contain $E(\rB_n,R)$, where $R$ is a subring of a commutative ring $A$.

The map $\rho\colon G(\rB_n)\to G(\rC_n)$ is constructed by restriction of the natural action of $G(\rB_n)$ on $(\rB_n, \varpi_1)=\langle e_1,\ldots,e_n,e_0,e_{-n},\ldots,e_{-1}\rangle$ onto $\langle e_1,\ldots,e_n,e_{-n},\ldots,e_{-1}\rangle$ (here $\langle e_0\rangle$ is the zero-weight subspace), see the end of the proof of \cite[Theorem~11.3.2(ii)]{CarterLie}. This restriction results in mapping
\begin{align*}
& x_\alpha(t)_{(\rB_n,\varpi_1)} \longmapsto x_{\alpha^\vee}(t)_{(\rC_n,\varpi_1)},\quad \text{if $\alpha\in\rB_n$ is long,} \\
& x_\alpha(t)_{(\rB_n,\varpi_1)} \longmapsto x_{\alpha^\vee}(t^2)_{(\rC_n,\varpi_1)},\quad \text{if $\alpha\in\rB_n$ is short.}
\end{align*}

The map $\theta\colon G(\rC_n)\to G(\rB_n)$ is constructed in a similar manner, but by considering the representations $(\rC_n,\varpi_n)$ and $(\rB_n,\varpi_n)$.

The fundamental representation $(\rC_n, \varpi_n)$ is a subrepresentation of $\wedge^n V$, where $V$ is the natural $2n$-dimensional module for $\operatorname{Sp}_{2n}$. We fix the basis $e_1,\ldots,e_n,e_{-n},\ldots,e_{-1}$ of $V$, such that the elementary generators act as
\begin{align*}
& T_{ij}(\xi) = e+\xi e_{ij}\pm\xi e_{-j,-i}, \quad i\neq\pm j,\\
& T_{i,-i}(\xi) = e+\xi e_{i,-i}.
\end{align*}
Here $\xi\in R$, $i,j\in\{1,\ldots,-1\}$. Denote
\[ e_A = \bigwedge_{\mathclap{a\in A}} e_a \quad \text{for } A\subseteq\{1,\ldots,-1\}. \]
Here we do not care about the order of factors for we are only interested in the case of characteristic $2$. The vectors $e_A$ with $|A|=n$ form a basis of $\wedge^n V$. For $A\subseteq\{1,\ldots,-1\}$ denote $A^{ij} = A \cap \{ \pm i, \pm j \}$ and $S(A)=A\cap-A$. Define a linear operator
\[ X_+\colon \wedge^n V \to \wedge^{n-2} V\quad \text{by} \quad X_+ e_A = \sum_{\mathclap{\qquad a\in S(A),\ a>0}} e_{A\setminus\{\pm a\}}.\]
Then $V(\varpi_n)=\ker X_+$ (see \cite[Ch.~VIII, \S13.3.IV]{BourLie79}). Now let us express the action of the elementary generators on $\wedge^n V$. 
\begin{align*}
& T_{i,-i}(\xi) e_A = \bigwedge_{\mathclap{\quad a\in A\setminus\{-i\}}} e_a \wedge
(\underbrace{e_{-i}\pm\xi e_i}_{\text{if $-i\in A$}}) =
\begin{cases}
e_A \pm \xi e_{A\setminus\{-i\}\cup\{i\}}, & \text{if $-i\in A$ and $i\notin A$}, \\
e_A, & \text{if $i\in A$ or $-i\notin A$}.
\end{cases}\\
& T_{ij}(\xi) e_A = \bigwedge_{\mathclap{\quad a\in A\setminus\{-i,j\}}} e_a \wedge (\underbrace{e_j\pm\xi e_i}_{\text{if $j\in A$}}) \wedge (\underbrace{e_{-i}\pm\xi e_{-j}}_{\text{if $-i\in A$}}) = \\
& \hphantom{T_{ij}(\xi) e_A} =
\begin{dcases}
e_A \pm \xi e_{A\setminus\{j\}\cup\{i\}}, & \text{if } A^{ij} \in \big\{\{j\},\{\pm j\},\{-i,\pm j\} \big\}, \\
e_A \pm \xi e_{A\setminus\{-i\}\cup\{-j\}}, & \text{if } A^{ij} \in \big\{ \{-i\},\{\pm i\},\{\pm i,j\} \big\}, \\
\begin{multlined}[b][12em]
e_A \pm \xi e_{A\setminus\{j\}\cup\{i\}} \pm \\
\pm \xi e_{A\setminus\{-i\}\cup\{-j\}} \pm \\
\pm \xi^2 e_{A\setminus\{-i,j\}\cup\{i,-j\}},
\end{multlined}& \text{if } A^{ij} = \{-i,j\}, \\
e_A & \text{otherwise.}
\end{dcases}
\end{align*}

We will now show that $U = \langle e_A,\ S(A)\neq\varnothing \rangle \cap \ker X_+$ is invariant in characteristic $2$.

For an element $u=\sum_k \alpha_k e_{A_k}$ and $B\subset \{1,\ldots,-1\}$, $|B|=n-2$ denote
\[ Y(u,B) = \sum_k \alpha_k,\quad \text{where the sum is over such $k$ that $A_k=B\cup\{\pm a\}$ for some $a$.} \]

An element $u=\sum_i \alpha_i e_{A_i}$ lies in $\ker X_+$ if for every $B\subset\{1,\ldots,-1\}$, $|B|=n-2$ the sum $Y(u, B)$ vanishes. Consider $u=\sum_k \alpha_k e_{A_k}$, where $S(A_k)\neq\varnothing$ for every $k$. Then
\[ T_{i,-i}(\xi)u = u + \xi \sum_{\mathclap{\qquad\quad k\colon -i\in A_k,\ i\notin A_k}} \alpha_k e_{A_k\setminus\{-i\}\cup\{i\}} = u+\xi u', \]
and all the summands in $u'$ are of the form $\alpha e_A$ with $S(A)\neq\varnothing$. Hence $T_{i,i}(\xi)u\in U$.
\begin{align*}
T_{ij}(\xi)u & =
\begin{aligned}[t]
u & + \xi \sum_{\mathclap{\qquad\qquad\qquad k\colon A_k^{ij} = \{j\},\{\pm j\},\{-i,\pm j\},\{-i,j\}}} \alpha_k e_{A_k\setminus\{j\}\cup\{i\}} + \\
& + \xi \sum_{\mathclap{\qquad\qquad\qquad k\colon A_k^{ij} = \{-i\},\{\pm i\},\{\pm i,j\},\{-i,j\}}} \alpha_k e_{A_k\setminus\{-i\}\cup\{-j\}} + \\
& + \xi^2 \sum_{\mathclap{k\colon A_k^{ij}=\{-i,j\}}} \alpha_k e_{A_k\setminus\{-i,j\}\cup\{i,-j\}} \equiv
\end{aligned} \\
& \equiv
\xi \sum_{\hspace{-2em}\mathclap{k\colon\substack{\mathrlap{A_k^{ij} =\{\pm j\}}\\ \mathrlap{S(A_k\setminus\{j\})=\varnothing}}}} \alpha_k e_{A_k\setminus\{j\}\cup\{i\}} +
\xi \sum_{\hspace{-2em}\mathclap{k\colon\substack{\mathrlap{A_k^{ij} = \{\pm i\}}\\\mathrlap{S(A_k\setminus\{-i\})=\varnothing}}}} \alpha_k e_{A_k\setminus\{-i\}\cup\{-j\}}
\pmod{U}
\end{align*}
Since $T_{ij}(\xi)u\in\ker X_+$, to each summand of $u$ of the form $\alpha_k e_{A_k}$ with $A_k=B\cup\{\pm j\}$, where $S(B)=\varnothing$, there corresponds another summand $\alpha_m e_{A_m}$, $A_m=B\cup\{\pm i\}$ with the same $B$, otherwise $Y(u,B)\neq0$. But then $A_k\setminus\{j\}\cup\{i\}=B\cup\{i,-j\}=A_m\setminus\{-i\}\cup\{-j\}$, so the basis elements involved in each of the two sums above coincide. Now note that since $U$ is defined by linear equations on $e_A$ with coefficients in $\mathbb{F}_2$, one can find a basis of $U$ that consists of $\mathbb{F}_2$-linear combinations of $e_A$, thus one can assume that all $\alpha_k$ in the above sums are actually $1$, and so this sum becomes zero in characteristic $2$.

Consider the quotient space $\ker X_+/U$. This is an $\operatorname{Sp}_{2n}$-module of dimension $2^n$ with the basis $e_A+U$, $S(A)=\varnothing$. The action of the induced root element operators $\overline{T}$ is described as
\begin{align*}
& \overline{T_{i,-i}}(\xi)e_A =
\begin{cases}
e_A + \xi e_{A\setminus\{-i\}\cup\{i\}} + U, & \text{if } -i\in A, \\
e_A + U & \text{otherwise.}
\end{cases} \\
& \overline{T_{ij}}(\xi)e_A = 
\begin{cases}
e_A + \xi^2 e_{A\setminus\{-i,j\}\cup\{i,-j\}} + U, & \text{if } -i,j\in A, \\
e_A + U & \text{otherwise.}
\end{cases}
\end{align*}
These are the formulas for the action of the elementary root unipotents in the spin representation of $\rB_n$, where the long root unipotents act by squares.

To see that this mapping does indeed send an arbitrary element $g$ of $\Sp(2n, R)$ to an element of $\Spin(2n+1, R)$, one has to check that the operator $\tilde{g}=\overline{\wedge^ng}$, acting on $\ker X_+/U$, is an element of $\Spin(2n+1, R)$ in its $2^n$-dimensional spin representation. The action of $\tilde{g}$ is given by the formula
\[ \tilde{g} (e_A+U) = \bigwedge_{\mathclap{a\in A}} ge_a + U. \]
Denote $M=\langle e_1,\ldots,e_n\rangle$ (so that $V=H(M)$ in the natural way), then one can identify $\ker X_+/U$ and $\wedge M$ by $e_A+U\mapsto e_{A\cap\{1,\ldots,n\}}$ and use the canonical isomorphism $C^+(2n+1) = C(H(M)) \cong \End(\wedge M)$ (see~\cite[Theorem~2.4]{BassCliff}). The latter isomorphism maps $e_i\in M$ to $(y\mapsto e_i\wedge y)$, which is, in turn, identified with the mapping
\[ s_i\colon e_A+U \longmapsto
\begin{cases}
e_{A\setminus\{-i\}\cup\{i\}}+U, & \text{if } -i\in A, \\
U, & \text{if } i\in A
\end{cases} \]
from $\End(\ker X_+/U)$. An element $f\in M^*=\langle e_{-n},\ldots,e_{-1}\rangle$ is mapped to the unique (anti)derivation $d_f$ of degree $-1$ prolonging $f$, in particular,
\[ d_{e_{-i}} (e_B) = \sum_{c\in B} e_{B\setminus\{c\}} d_{e_{-i}}(e_c) =
\begin{cases}
e_{B\setminus\{i\}}, & \text{if } i\in B, \\
0, & \text{if } i\notin B.
\end{cases} \]
This operator acts on $\ker X_+/U$ as
\[ s_{-i}\colon e_A+U \longmapsto
\begin{cases}
e_{A\setminus\{i\}\cup\{-i\}}+U, & \text{if } i\in A, \\
U, & \text{if } -i\in A.
\end{cases} \]

By definition \cite[\S~3.1]{BassCliff}, the special Clifford group of a quadratic module $P$ is
\[ \SCliff(P)=\{u\in C^+(P)^* \mid \pi(u)P\leqslant P\}. \]
Here $\pi(u)a = \pm uau^{-1}$. Thus to show that $\tilde{g}\in\SCliff(2n+1)$, one has to check that conjugation by $\tilde{g}$ preserves the linear span $\langle s_1,\ldots,s_n,\operatorname{id},s_{-n},\ldots,s_{-1}\rangle\leqslant\End(\ker X_+/U)$. The $\Sp(2n,R)$-invariance of this subspace is checked by a straightforward computation. Namely, denote $q=\overline{T_{i,-i}}(\xi)$ and $r=\overline{T_{ij}}(\xi)$, then
\[ {}^q\!s_j =
\begin{cases}
s_{-i}+\xi\operatorname{id}+\xi^2s_i, & \text{if } j=-i, \\
s_j & \text{otherwise},
\end{cases}
\qquad
{}^r\!s_k =
\begin{cases}
s_{-i}+\xi^2s_{-j}, & \text{if } k=-i, \\
s_j+\xi^2s_i, & \text{if } k=j, \\
s_k & \text{otherwise}.
\end{cases} \]
The Spin group is defined \cite[\S~3.2]{BassCliff} as the kernel of the norm map $N\colon\SCliff(P)\to R^*$. Here $N(x)=x\overline{x}$, and $x\mapsto\overline{x}$ is the antiautomorphism of $C(P)$ that extends the identity map on $P$. On $C^+(2n+1)\cong\End(\ker X_+/U)$ it can be defined by $N(x)=Jx^tJ$, where $J_{AB}=\delta_{A,-B}$ (with respect to $e_A+U$, $S(A)=\varnothing$). Consider $K\in\End{\wedge^n V}$, defined (with respect to $e_A$) by $K_{AB}=\delta_{A,\overline{B}}$, $\overline{B}=\{1,\ldots,-1\}\setminus B$. Then if follows from the Laplace expansion in multiple rows that
\[ \wedge^n g \cdot K \cdot (\wedge^n g)^t \cdot K \equiv \det(g)\cdot \operatorname{id}_{\wedge^n V} \pmod{2}, \]
and hence $N(\tilde{g})=\tilde{g}J\tilde{g}^t\!J=\det(g)=1$, so $\theta\colon g\mapsto\tilde{g}$ maps $\Sp(2n,R)$ to $\Spin(2n+1,R)$.

\section{Tits mixed groups}
The maps $\rho$ and $\theta$, constructed in the previous section, allow one to give a simple definition of the Tits mixed groups over commutative rings. Usually one considers a pair of fields $E\leqslant F$ of characteristic $p$, the part of an infinite chain of fields
\[ \ldots \leqslant F^{p^2} \leqslant E^p \leqslant F^p \leqslant E \leqslant F \leqslant E^{1/p} \leqslant F^{1/p} \leqslant \ldots \]
The mixed group is then defined \cite[\S~10.3.2]{TitsBuildings} (see also \cite{NaertTwisting}) as
\[ G(\Phi, E, F) = \left\langle x_\alpha(\xi) \colon \alpha\in\Phi,\
\begin{aligned}
& \xi\in E \text{ if $\alpha$ is long,} \\
& \xi\in F \text{ if $\alpha$ is short}
\end{aligned} \right\rangle. \]
Over a ring this defines the elementary subgroup, but not the ambient group. Note that $\theta(G(\rC_n, F)) = G(\rB_n, F^2, F)$ and that $G(\rB_n, F^2, E) = G(\rB_n, E) \cap G(\rB_n, F^2, F)$. Now since $G(\rB_n, E, F) \cong G(\rC_n, F^2, E)$, also $G(\rB_n, F^2, E) \cong G(\rC_n, E^2, F^2) = \varphi(G(\rC_n, E, F))$. This definitions of mixed group are easily extended to the groups over rings.

For the mixed groups of types $\rG_2$ or $\rF_4$ everything is even easier. For example, in case $\Phi=\rG_2$ the reduction map $\vartheta\colon G(\rG_2)_3\to \End(V(\mu))$ sends an element of $G(\rG_2)_3=G_{\mathrm{ad}}(\rG_2)$ to an element of $G(\rG_2)_\mu$, since the preserved forms $B$ and $T$ on $V$ are simply the reductions modulo $3$ of the Killing form $\kappa(u,v)=\operatorname{tr}(\operatorname{ad}u\cdot\operatorname{ad}v)$ and of the unique $G_{\mathrm{ad}}(\rG_2)$-invariant (alternating) trilinear form $(u,v,w)\mapsto\kappa([u,v],w)$ on the adjoint module. Then one defines $G(\rG_2, F^3, E) = G(\rG_2, E) \cap \vartheta(G(\rG_2, F))$.

\bibliographystyle{amsalpha}
%\printbibliography
\bibliography{suzuki-ree}

\providecommand{\bysame}{\leavevmode\hbox to3em{\hrulefill}\thinspace}
\providecommand{\MR}{\relax\ifhmode\unskip\space\fi MR }
\providecommand{\MRhref}[2]{%
  \href{http://www.ams.org/mathscinet-getitem?mr=#1}{#2}
}
\providecommand{\href}[2]{#2}
\begin{thebibliography}{PSV98}

\bibitem[{\CYRN}{\CYRS}18]{NuzSteBnCnSubring}
Я.~Н. {\CYRN}ужин and A.~В. {\CYRS}тепанов,
  \emph{Подгруппы групп {Ш}евалле типов ${B}_n$ и
  ${C}_n$, содержащие группу над подкольцом, и
  связанные с ними ковры}, preprint (2018).

\bibitem[Asc87]{AschbG2Trilinear}
M.~Aschbacher, \emph{Chevalley groups of type ${G}_2$ as the group of a
  trilinear form}, J. Algebra \textbf{109} (1987), no.~1, 193--259.

\bibitem[Asc88]{AschbMultilinear}
\bysame, \emph{Some multilinear forms with large isometry groups}, Geom.
  Dedicata \textbf{25} (1988), no.~1-3, 417--465.

\bibitem[Bas74]{BassCliff}
H.~Bass, \emph{Clifford algebras and spinor norms over a commutative ring},
  Amer. J. Math. \textbf{96} (1974), no.~1, 156--206.

\bibitem[Bou08]{BourLie79}
N.~Bourbaki, \emph{Lie groups and {L}ie algebras: chapters 7-9}, vol.~7,
  Springer Science \& Business Media, 2008.

\bibitem[BS17]{BakSteSymplecticSubring}
A.~Bak and A.~Stepanov, \emph{Subring subgroups of symplectic groups in
  characteristic $2$}, St. Petersburg Math. J. \textbf{28} (2017), no.~4,
  465--475.

\bibitem[Car89]{CarterLie}
R.~W. Carter, \emph{Simple groups of {L}ie type}, Wiley Classics Library, John
  Wiley \& Sons, Inc., New York, 1989, Reprint of the 1972 original, A
  Wiley-Interscience Publication.

\bibitem[Nae17]{NaertTwisting}
K.~Naert, \emph{Twisting and mixing}, arXiv preprint arXiv:1703.03794 (2017).

\bibitem[PSV98]{PloSemVavVBRA}
E.~Plotkin, A.~Semenov, and N.~Vavilov, \emph{Visual basic representations: an
  atlas}, Internat. J. Algebra Comput. \textbf{8} (1998), no.~1, 61--95.

\bibitem[Tit60]{TitsSuzukiRee}
J.~Tits, \emph{Les groupes simples de {S}uzuki et de {R}ee}, S{\'e}minaire
  Bourbaki \textbf{13} (1960), no.~210, 1--18.

\bibitem[Tit62]{TitsOvoides}
\bysame, \emph{Ovo{\"i}des et groupes de {S}uzuki}, Archiv der Mathematik
  \textbf{13} (1962), no.~1, 187--198.

\bibitem[Tit74]{TitsBuildings}
\bysame, \emph{Buildings of spherical type and finite ${BN}$-pairs}, Lectures
  Notes in Mathematics \textbf{386} (1974).

\end{thebibliography}
\end{document}